\documentclass[11p]{amsart}

\usepackage{amsfonts}
\usepackage{amsthm}
\usepackage{amssymb}
\usepackage{amsmath}
\usepackage{url, hyperref}

\newtheorem{theo}{\bf Theorem}[section]
\newtheorem{lemma}{\bf Lemma}[section]
\newtheorem{coro}{\bf Corollary}[section]

\newcommand{\C}{{\mathcal C}}
\newcommand{\D}{{\mathcal D}}

\newcommand{\Z}{{\mathbb Z}}
\newcommand{\N}{{\Bbb N}}
\newcommand{\Q}{{\Bbb Q}}
\newcommand{\R}{{\mathbb R}}

\newcommand{\bea}{\begin{eqnarray*}}
\newcommand{\eea}{\end{eqnarray*}}
\newcommand{\be}{\begin{eqnarray}}
\newcommand{\ee}{\end{eqnarray}}

\newcommand{\vol}{\mbox{vol}\,}

\newcommand{\diam}{\mbox{\rm diam}\,}

\newcommand{\frob}{\mathrm{F}}

\newcommand{\modulo}{\,\mathrm{mod}\;}
\newcommand{\cone}{\mathrm{cone}}
\newcommand{\sg}{\mathrm{Sg}}
\newcommand{\gap}{\mathrm{gap}}
\newcommand{\Lat}{{\mathcal L}}

\numberwithin{equation}{section}

\begin{document}


\title{On the lattice programming gap of the group problems}
\author{Iskander Aliev}
\address{School of Mathematics, Cardiff University, Cardiff, UK, email: alievi@cardiff.ac.uk}



%
\begin{abstract}
Given a full-dimensional lattice $\Lambda\subset \Z^k$ and a cost vector $l\in \Q^k_{>0}$, we are concerned with
the family of the group problems
\be\label{abstract_problem}
\min\{ l\cdot x: x \equiv r (\modulo \Lambda), x\ge 0\}\,,\;\;\; r\in\Z^k.
\ee
The {\em lattice programming gap} $\gap(\Lambda,l)$ is the largest value of the minima in (\ref{abstract_problem})
as $r$ varies over $\Z^k$. We show that computing the  lattice programming gap is NP-hard when $k$ is a part of input.
We also obtain lower and upper bounds for $\gap(\Lambda,l)$ in terms of $l$ and the determinant of $\Lambda$.


\end{abstract}





\maketitle

\section{Introduction and statement of results}

Consider the integer programming problem
\be
\min\{ c\cdot x: Ax=b, x\ge 0\,,x\;\mbox{is integer}\}\,.
\label{initial_IP}
\ee
Gomory \cite{Gomory_group_relaxation} defined a group relaxation of $(\ref{initial_IP})$ as follows. Let $B$ and $N$ be the index sets of basic and nonbasic variables for an optimal basic solution to
the linear programming relaxation $\min\{ c\cdot x: Ax=b, x\ge 0\}$ of (\ref{initial_IP}).
Then the problem (\ref{initial_IP}) can be written as
\be \begin{split}
\min\{ c_B\cdot x_B+c_N\cdot x_N: A_Bx_B+A_Nx_N=b\,,x_B, x_N\ge 0\,,\\ x_B, x_N\;\mbox{are integer}\}\,
\label{initial_IP_rewritten}\end{split}
\ee
and a relaxation of (\ref{initial_IP_rewritten}) is obtained by removing the restriction $x_B\ge 0$:
\be \begin{split}
\min\{ c_B\cdot x_B+c_N\cdot x_N: A_Bx_B+A_Nx_N=b\,, x_N\ge 0\,,\\ x_B, x_N\;\mbox{are integer}\}\,.
\label{initial_group_relaxation}\end{split}
\ee
Hence (\ref{initial_group_relaxation}) is a lower bound for (\ref{initial_IP}) and it can be used in any branch and bound procedure.

The constraints $A_Bx_B+A_Nx_N=b$ in (\ref{initial_group_relaxation}) can be written in the equivalent form
$x_B=A_B^{-1}b-(A_B^{-1}A_N)x_N$.
Thus, given any nonnegative integral vector $x_N$, the vector $x_B$ is integer if and only if
$(A_B^{-1}A_N)x_N\equiv A_B^{-1}b (\modulo 1)$.
Setting $c_N'=c_N-c_BA_B^{-1}A_N$, we can rewrite (\ref{initial_group_relaxation}) as
\be
\min\{ c_N'\cdot x_N: (A_B^{-1}A_N)x_N\equiv A_B^{-1}b (\modulo 1)\,, x_N\ge 0\,, x_N\;\mbox{is integer}\}\,.
\label{group_relaxation}
\ee
The program (\ref{group_relaxation}) is called the {\em Gomory's group relaxation} for (\ref{initial_IP}).

In this paper we fix a cost vector $c\in \Q^n$ and for a matrix $A\in \Z^{d\times n}$ of rank $d$ and $b\in \sg(A)=\{Au: u\in \Z^n_{\ge 0}\}$ consider the integer program
\bea
IP_{c}(A,b)=\min\{ c\cdot x: Ax=b, x\in \Z^n_{\ge 0}\}\,.
\eea
For simplicity, we assume that the cone $\cone(A)=\{Ax: x\ge 0\}$ is pointed and that the subspace $A^{\bot}=\{x\in \R^n: Ax=0\}$, the kernel of $A$,
intersects the nonnegative orthant $\R^n_{\ge 0}$ only at the origin. This assumption guarantees that $IP_{c}(A,b)$ is bounded for all $b\in \sg(A)$.

Consider the $(n-d)$-dimensional lattice $\Lat(A)=A^{\bot}\cap \Z^n$.  The program $IP_{c}(A,b)$ is equivalent to the lattice program
\be
\min\{ c\cdot x: x \equiv u (\modulo \Lat(A)), x\ge 0\}\,,
\label{lattice_IP}
\ee
where $u$ is any integer solution of the equation $Ax=b$.

A subset $\tau$ of $\{1,\ldots,n\}$ partitions $x\in \R^n$ as $x_{\tau}$ and $x_{\bar \tau}$, where $x_{\tau}$ consists of the entries indexed by $\tau$
and $x_{\bar \tau}$ the entries indexed by the complimentary set ${\bar \tau}$. Similarly, the matrix $A$ is partitioned as $A_\tau$ and $A_{\bar \tau}$.
Let $\tau$ be the set of indices of the basic variables for an optimal solution to the linear relaxation
$LP_{c}(A,b)=\min\{ c\cdot x: Ax=b, x\ge 0\}$
of the integer program $IP_{c}(A,b)$.
Let $\pi_{\tau}$ be the projection map from $\R^n$ to $\R^{n-d}$ that forgets all  coordinates indexed by ${\tau}$ and let $\Lambda(A)=\pi_{\tau}(\Lat(A))$.
The lattices $\Lat(A)$ and $\Lambda(A)$ are isomorphic (see e.g. Section 2 in \cite{RT_structure}) and the Gomory's group relaxation for $IP_{c}(A,b)$ is equivalent to the {\em lattice program}
\be
\min\{ c_{\bar \tau}'\cdot x: x \equiv u_{\bar \tau} (\modulo \Lambda(A)), x\ge 0\}\,,
\label{Gomory_lattice_relaxation}
\ee
where $c_{\bar \tau}'=c_{\bar \tau}-c_{\tau}A_{\tau}^{-1}A_{\bar \tau}$. Note that the vector $c_{\bar \tau}'$  is nonnegative. For simplicity
we will consider in this paper the {\em generic} case, when all entries of $c_{\bar \tau}'$ are positive.

The group relaxations can be defined for various sets of variables. Wolsey \cite{W} introduced the {\em extended group relaxations} obtained by dropping nonnegativity restrictions on the variables indexed by each subset of $\tau$. Ho\c{s}ten and Thomas \cite{HT} studied the set of all group relaxations obtained by dropping nonnegativity restrictions on the variables indexed by each face of a polyhedral complex associated
with $A$ and $c$.
For further details on the classical theory of group relaxations we refer the reader to \cite{Johnson_book} and \cite{AWW}.

In this paper we will consider the group relaxations in the following general form.
For a fixed cost vector $l\in \Q^k_{>0}$, a $k$-dimensional lattice $\Lambda\subset\Z^k$ and $r\in \Z^k$
we are concerned with the lattice program (also referred to as the {\em group problem})
\be\begin{split}
\min\{ l\cdot x: x \equiv r (\modulo \Lambda), x\ge 0\}\,.
\end{split}
\label{generic_group_relaxation}
\ee
Let $m(\Lambda,l,r)$ denote the value of the minimum in (\ref{generic_group_relaxation}).
We are interested in the {\em lattice programming gap} $\gap(\Lambda,l)$ of (\ref{generic_group_relaxation}) defined as
\be\begin{split}
\gap(\Lambda,l)=\max_{r\in \Z^k}m(\Lambda,l,r)\,.
\end{split}
\label{maximum_gap}
\ee

The lattice programming gaps were introduced and studied for sublattices of all dimensions in $\Z^k$ by Ho\c{s}ten and Sturmfels \cite{HS}.
The algebraic and algorithmic results on the lattice programming gaps obtained in \cite{HS} have applications to the statistical theory of multidimensional
contingency tables.

For fixed $k$ the value of $\gap(\Lambda,l)$ can be computed in polynomial time (see Section 3 in \cite{HS} and \cite{Fritz}).
The first result of this paper shows that computing $\gap(\Lambda,l)$ is NP-hard when $k$ is a part of input.
\begin{theo}\label{Gap_Frobi}
Computing $\gap(\Lambda, l)$ is NP-hard.
\end{theo}
The proof of Theorem \ref{Gap_Frobi} is based on a connection between  the lattice programming gaps and the Frobenius numbers.
Computing Frobenius numbers is NP-hard due to the well-known result of Ram\'{\i}rez Alfons\'{\i}n \cite{Alf1}.

Our next goal is to obtain the lower and upper  bounds for $\gap(\Lambda,l)$ in terms of the parameters of the
lattice program (\ref{generic_group_relaxation}). The bounds on the lattice programming gap provide bounds on the possible objective solutions when considering Gomory's group relaxation type problems. We show that the obtained lower bound is optimal and that the upper bound has the optimal order. The proofs are based on recent results of Marklof and Strombergson \cite{MS}
on the diameters of circulant graphs and on the estimates of Fukshansky and Robins \cite{Lenny} for the Frobenius numbers.

For a given closed bounded convex set $K$ with nonempty interior in $\R^k$ and a $k$-dimensional lattice $\Lambda\subset \R^k$,
the {\em covering radius} of $K$ with respect to $\Lambda$ is defined as $\rho(K,\Lambda)= \min\{r>0: r K + \Lambda =\R^k\}$.
Let $X_k$ be the set of all $k$-dimensional lattices $\Lambda\subset \R^k$ of determinant one, let $\Delta=\{x\in\R^k_{\ge 0}: \sum_{i=1}^k x_i\le 1\}$ be the standard $k$-dimensional simplex and let $\rho_k=\inf_{\Lambda \in X_k} \rho(\Delta, \Lambda)$. 
We obtain the following optimal lower bound for $\gap(\Lambda,l)$.
\begin{theo}
\begin{itemize}
\item[(i)] For any $l\in \Q^k_{>0}$, $k\ge 2$, and any $k$-dimensional lattice $\Lambda\subset \Z^k$
\be
\gap(\Lambda,l)\ge \rho_k (\det(\Lambda) l_1 \cdots l_k)^{1/k}-\sum_{i=1}^k l_i\,.
\label{optimal_bound}
\ee
\item[(ii)] For any $c\in \Q^{k+1}_{>0}$, $k\ge 2$, and any $\epsilon>0$, there exists a matrix $A\in \Z^{1\times (k+1)}$  such that
for all $b\in \sg(A)$ the knapsack problem $LP_c(A,b)$ has a unique solution with nonbasic variables indexed by $\sigma=\{1,\ldots,k\}$ and for
$l=c'_{\sigma}$
\be
\gap(\Lambda(A),l)< (\rho_k+\epsilon) (\det(\Lambda(A)) l_1 \cdots l_k)^{1/k}-\sum_{i=1}^k l_i\,.
\label{optimality}
\ee
Furthermore, there exists $b'\in \sg(A)$ such that the optimal value of $IP_c(A,b')$ is equal to $\gap(\Lambda(A),l)+ c_{\bar\sigma}A_{\bar\sigma}^{-1}b'$.
\end{itemize}
\label{lemma_optimal_bound}
\end{theo}

The only known values of $\rho_k$ are $\rho_1=1$ and $\rho_2=\sqrt{3}$ (see  \cite{Fary}).
It was proved in \cite{AG}, that $\rho_k>(k!)^{1/k}$. Thus we obtain the following estimate.

\begin{coro} For any $l\in \Q^k_{>0}$, $k\ge 2$, and any $k$-dimensional lattice $\Lambda\subset \Z^k$
\be
\gap(\Lambda,l)> (k! \det(\Lambda) l_1 \cdots l_k)^{1/k}-\sum_{i=1}^k l_i\,.
\label{coro_lower_bound}
\ee
\end{coro}
For sufficiently large $k$ the bound (\ref{coro_lower_bound}) is not far from being optimal. Indeed, $\rho_k\le (k!)^{1/k}(1+O(k^{-1}\log k))$ (cf. \cite{DF}).

Group relaxations provide the lower bounds for integer programs $IP_{c}(A,b)$. From this viewpoint, part (i) of Theorem \ref{lemma_optimal_bound}  and Corollary
\ref{coro_lower_bound} estimate the largest possible value that such a bound can take.  Part (ii) of Theorem \ref{lemma_optimal_bound}
also shows that the obtained result is optimal in the case of knapsack problems.

Let $|\cdot|$ denote the Euclidean norm and let $\gamma_k$ be the $k$-dimensional Hermite constant
(see i.e. Section IX.7 in \cite{Cassels}).
We give the following upper bound for $\gap(\Lambda,l)$ (and hence for the minimum in (\ref{Gomory_lattice_relaxation})).
\begin{theo} For any $l\in \Q^k_{>0}$, $k\ge 2$, and any $k$-dimensional lattice $\Lambda\subset \Z^k$
\be
\gap(\Lambda,l)\le  \frac{k\gamma_k^{k/2}\det(\Lambda)(\sum_{i=1}^k l_i + |l|)}{2}-\sum_{i=1}^k l_i\,.
\label{upper_bound}
\ee
\label{lemma_upper_bound}
\end{theo}
The known exact values of $\gamma_k^k$ are $1$, $4/3$, $2$, $4$, $8$, $64/3$, $64$, $256$ (Sloan's sequence A007361 in \cite{OEIS}). By a result of Blichfeldt  (see, e.g. \cite{GrLek}) $\gamma_k\le 2\left(\frac{k+2}{\sigma_k}\right)^{2/k}$,
where $\sigma_k$ is the volume of the unit $k$-ball;  thus
$\gamma_k=O(k)$. The precision of the bound (\ref{upper_bound}) depends on the estimates for the covering radius of a simplex,
associated with the cost vector $l$, with respect to the lattice $\Lambda$. It follows from results in \cite[Section 6]{AH} that the order $\gap(\Lambda,l)= O_{k,l}(\det(\Lambda))$,
where the constant depends on $k$ and $l$, cannot be improved.

A widely used approach (see e.g. \cite{AEGJ}) is to consider a group relaxation induced by a single row $i$: $\sum_{j\in N} \hat{a}_{ij}x_j \equiv \hat{b}_i (\modulo 1)$ of the matrix constraint in (\ref{group_relaxation}). Here we may assume that all $\hat{a}_{ij}$ and $\hat{b}_i$ are rational numbers from $[0,1)$
with common denominator $D=|\det(B)|$. Thus, multiplying by $D$, we get the constraint $\sum_{j\in N} (D\hat{a}_{ij})x_j \equiv D\hat{b}_i (\modulo D)$. Set $k=|N|$,  $A=(D\hat{a}_{i1}, \ldots, D\hat{a}_{ik}, D)\in \Z^{1\times(k+1)}$ and $\Lambda=\pi_{\{k+1\}}(\Lat(A))$. We may assume that $l=c'_{\bar\tau}\in \Q^k_{>0}$, where $\tau$ is the set of indices of basic variables. Then for any integer solution $r\in \Z^k$ of $r\cdot \pi_{\{k+1\}}(A)\equiv D\hat{b}_i (\modulo D)$ the group relaxation induced by the row $i$ can be written in the form (\ref{generic_group_relaxation}). Thus all bounds derived in this paper can be applied to the group relaxation induced by a selected row of (\ref{group_relaxation}). Note that in this special case the lattice programming gap $\gap(\Lambda,l)$ can be associated with the diameter of a directed circulant graph (see \cite{MS} for details). Furthermore, the results of \cite{MS}  show that the lower bound (\ref{optimal_bound}) is a good predictor for the value of $\gap(\Lambda,l)$ for a `typical' $\Lambda$.

\section{$\gap(\Lambda,l)$ and diameters of quotient lattice graphs}

Assume for the rest of the paper $k\ge 2$. Following notation from \cite{MS}, let  $LG^+_k = (\Z^k, E)$ be the standard directed lattice graph with vertex set $\Z^k$. The edge set $E$ consists of all directed edges
$(x, x+e_j)$, where $x\in \Z^k$ and $e_1, \ldots, e_k$ are the standard basis vectors. 
%
%
Let $\Lambda$ be a $k$-dimensional sublattice of $\Z^k$.
We define the quotient lattice graph $LG_k^+/\Lambda$ as the digraph with vertex set $\Z^k/\Lambda$ and
the edge set $\{(x+\Lambda, x+e_j+\Lambda): x\in \Z^k, j=1, \ldots, k\}$.
Given cost vector $l\in \Q^k_{>0}$, we define the distance from vertex $x+\Lambda$ to $y+\Lambda$ in $LG^+_k/\Lambda$  as
\bea
d_{LG^+_k/\Lambda}(x+\Lambda,y+\Lambda)=
\min_{z\in (y-x+\Lambda)\cap \Z^k_{\ge 0}} l\cdot z \,.
\eea
%
The diameter of  $LG_k^+/\Lambda$
is given by $\diam(LG_k^+/\Lambda)= \max_{y\in \Z^k/\Lambda} d_{LG^+_k/\Lambda}(0+\Lambda,y+\Lambda)$.
%
%
Since for any $y\in\Z^{k}$
\bea\begin{split}
d_{LG^+_k/\Lambda}(0+\Lambda,y+\Lambda)
=\min\{ l\cdot x: x \equiv y (\modulo \Lambda), x\ge 0\}\,,
\end{split}
\eea
we obtain the following expression (cf. \cite{Gomory_group_relaxation}).
\begin{lemma}
$\gap(\Lambda,l)= \diam(LG_{k}^+/\Lambda)\,.$
\label{lattice_gap_diameter}
\end{lemma}

\section{$\gap(\Lambda,l)$ and the covering radius of a simplex}

Given cost vector $l\in \Q^k_{>0}$, let $\Delta_{l}  =\left\{ x \in\R_{\ge 0}^{k}: \, l\cdot x\le 1\right\}$.
Then the following result holds.
\begin{lemma}
$\gap(\Lambda,l)=\rho(\Delta_{l}, \Lambda)-\sum_{i=1}^{k} l_i\,.
$
\label{gap_via_simplex}
\end{lemma}
\begin{proof}
The result follows from Lemma \ref{lattice_gap_diameter} and results of \cite{MS}. For completeness we give here a detailed proof.
Where possible, we keep the notation from \cite{MS} for convenience of the reader.

Let $\Lambda$ be a $k$-dimensional sublattice of $\Z^k$.  Consider the continuous torus $\R^k/\Lambda$. We can define the distance $d_{\R^k/\Lambda}$ between any two points $x+\Lambda$ and $y+\Lambda$ on $\R^k/\Lambda$ as
\bea
d_{\R^k/\Lambda}(x+\Lambda,y+\Lambda)=\min_{z\in (y-x+\Lambda)\cap\R^k_{\ge 0}} l\cdot z\,.
\eea
By the directed diameter of $\R^k/\Lambda$ we understand $\diam_l^+(\R^k/\Lambda)= \sup_{y\in \R^k/\Lambda} d_{\R^k/\Lambda}(0+\Lambda,y+\Lambda)$.
It follows from the proof of Lemma 3 in \cite{MS} that $\diam(LG_k^+/\Lambda)= \diam_l^+(\R^k/\Lambda)-\sum_{i=1}^k l_i$.  Then by Lemma \ref{lattice_gap_diameter} we can express $\gap(\Lambda,l)$ as
\be
\gap(\Lambda,l)=\diam_l^+(\R^k/\Lambda)-\sum_{i=1}^k l_i\,.
\label{gap_via_diameter}
\ee

Next, define the lattice $\Gamma(\Lambda,l) = \Lambda\,{\rm diag}(\Pi^{-1/k}l_1, \ldots,\Pi^{-1/k}l_k)$,
where $\Pi=\det(\Lambda)l_1 \cdots l_k$. Then for $e=(1,\ldots,1)\in \Z^k$ we have
\be
\diam_l^+(\R^k/\Lambda)=\Pi^{1/k}\diam_e^+(\R^k/\Gamma(\Lambda,l))\,.
\label{discrete_continuous}
\ee
By Lemma 4 in \cite{MS},
\be
\diam_e^+(\R^k/\Gamma(\Lambda,l))=\rho(\Delta, \Gamma(\Lambda,l))\,.
\label{diameter_radius}
\ee
Since the linear transform defined by the matrix ${\rm diag}(\Pi^{-1/k}l_1, \ldots,\Pi^{-1/k}l_k)$ maps $\Delta_l$ to $\Pi^{-1/k}\Delta$, we have
\be
\rho(\Delta, \Gamma(\Lambda,l))=\Pi^{-1/k}\rho(\Delta_{l}, \Lambda)\,.
\label{radii}
\ee
Combining (\ref{gap_via_diameter}), (\ref{discrete_continuous}), (\ref{diameter_radius}) and (\ref{radii}), we complete the proof of the lemma.
\end{proof}

\section{Proof of Theorem \ref{Gap_Frobi}}

We are concerned with the following problem:
\be\label{comput_gap}
\mbox{Given a } k\mbox{-dimensional lattice } \Lambda\subset \Z^k \mbox{ and } l\in \Q^k , \mbox{ compute }\gap(\Lambda,l)\,.
\ee
Here we suppose that the lattice $\Lambda$ is given by its basis.

Let $a$ be a positive integral $n$-dimensional primitive vector with $n=k+1$, i.e., $a=(a_1,\dots,a_{k+1})^t\in\Z_{>0}^{k+1}$ with $\gcd(a_1,\dots,a_{k+1})=1$.
The {\em Frobenius number} $\frob(a)$ is the largest number which cannot be
represented  as a nonnegative integral combination of the $a_i$'s.
The problem of computing $\frob(a)$ has been traditionally referred to as the
{\em Frobenius problem}.  This problem is NP-hard when $n$ is a part of input (Ram\'{\i}rez
Alfons\'{\i}n \cite{Alf1}).

Set $l_a=(a_1, \ldots, a_k)^t$ and $\Lambda_{a}=\{x\in \Z^k: a_1 x_1+\cdots +a_k x_k \equiv 0 \,(\modulo a_{k+1})\}$. 
By a celebrated result of Kannan \cite{Kannan} the Frobenius number can be expressed as
$
\frob(a)= \rho(\Delta_{l_a}, \Lambda_a)-\sum_{i=1}^{k+1} a_i\,.
$
Hence,  Lemma \ref{gap_via_simplex} with  $l=l_a$ implies
\be
\frob(a)=\gap(\Lambda_a, l_a)-a_{k+1}\,.
\label{gap_via_Frobi}
\ee
By Corollary 5.4.10 in \cite{GLS}, given integer vector $a$, a basis of  $\Lambda_a$ can be computed in polynomial time.
Therefore, the formula (\ref{gap_via_Frobi}) provides a polynomial time Turing reduction from the Frobenius problem to (\ref{comput_gap}).

\section{Proof of Theorem \ref{lemma_optimal_bound}}

Part (i).
By Lemma \ref{gap_via_simplex} and (\ref{radii}) we can write
\be
\gap(\Lambda,l)=\rho(\Delta, \Gamma(\Lambda,l))\Pi^{1/k}-\sum_{i=1}^k l_i\,.
\label{gap_via_gamma}
\ee
Since $\Gamma(\Lambda,l)\in X_k$, the inequality (\ref{optimal_bound}) now follows from the
definition of $\rho_k$.

Part (ii).
There exists $u=(p_1/q,\ldots,p_{k+1}/q)\in \Q^{k+1}_{>0}$ with $p_1, \ldots, p_{k+1},q\in \Z_{>0}$,
such that for  any $b\in \sg(qu^t)$ the linear relaxation $LP_{c}(qu^t,b)$ has a unique optimal solution with nonbasic variables indexed by $\sigma=\{1,\ldots,k\}$.
Let ${\mathfrak F}=\{x\in \R^{k+1}: 0<x_1<\ldots<x_{k+1}\}$. Changing the order of coordinates and perturbing $u$, if needed, we may assume that $u\in {\mathfrak F}$.
For $\epsilon>0$ let $\C_\epsilon=\{x\in \R^{k+1}: \left|u/|u|-x/|x|\right|<\epsilon\}$.
One can choose sufficiently small $\epsilon_0>0$ such that $\C_{\epsilon_0}\subset {\mathfrak F}$ and for any $v\in \C_{\epsilon_0}\cap\Z^{k+1}$  the linear relaxation  $LP_{c}(v^t,b)$ has a unique optimal solution with nonbasic variables indexed by $\sigma$ for any $b\in \sg(v^t)$.

Set $\D=\C_{\epsilon_0}\cap [0,1]^{k+1}$, $l=c'_{\sigma}$ and  $\widehat{\N}^{k+1}$ be the set of integral vectors in $\R^{k+1}$ with positive coprime coefficients (i.e., the greatest common divisor of all coefficients is one).
We can view $\Gamma(\Lambda(a^t),l)$ as an $X_k$-valued random variable defined by taking $a$ uniformly at random in $\widehat{\N}^{k+1}\cap T \D$ for some $T>0$.
Let $\mu_0$ be the $SL(k, \R)$ invariant probability measure on $X_k$. It was shown in Section 2.5 of \cite{MS} that, as $T\rightarrow \infty$, $\Gamma(\Lambda(a^t),l)$ converges in distribution to a random variable $L\in X_k$, taken according to $\mu_0$ (note that in notation of \cite{MS}, $\Gamma(\Lambda(a^t),l)$ corresponds to $L_{n,{\boldsymbol a}, {\boldsymbol l}}$). Furthermore, following Section 2.5 of \cite{MS}, the function $L\rightarrow \rho(\Delta, L)$ is continuous on $X_k$ and hence, by the continuous mapping theorem, %
\be \rho(\Delta, \Gamma(\Lambda(a^t),l))\xrightarrow{\;d\;} \rho(\Delta, L)\;\;\mbox{as}\;\; T\rightarrow\infty\,,
\label{conv_distr}
\ee
where $X \xrightarrow{\;d\;} Y$ denotes convergence in distribution.

Consider the complementary distribution function $P_k(R)=\mu_0(\{\Lambda\in X_k: \rho(\Delta, \Lambda)>R\})$ of $\rho(\Delta, L)$.
Then (\ref{conv_distr}) is equivalent with the statement that for any $R\ge 0$ we have
\be\label{limit_prob}\begin{split}
\lim_{T\rightarrow \infty}\frac{1}{\#(\widehat{\N}^{k+1}\cap T \D) }\#\{a\in \widehat{\N}^{k+1}\cap T \D: \rho(\Delta, \Gamma(\Lambda(a^t),l))<R  \}\\
=1-P_k(R)\,.
\end{split}\ee


It was proved in \cite{Marklof}
that $P_k(R)$ is continuous  for any fixed $k\ge 2$. It was also noticed in \cite{MS}, Remark 1.2 (see also \cite{Str}, p. 86) that
\be
P_k(R)=1 \;\mbox{for}\; 0\le R\le \rho_k,\;
\mbox{and}\; P_k(R)<1 \;\mbox{for} \;R>\rho_k.
\label{comp_distr}
\ee
By (\ref{comp_distr}) for any $\epsilon>0$ we have $P_k(\rho_k+\epsilon)<1$.
Hence, by (\ref{limit_prob}), 
for sufficiently large $T$ there exists a vector $a\in \widehat{\N}^{k+1}\cap T\D$  such that $\rho(\Delta, \Gamma(\Lambda(a^t),l))< \rho_k+\epsilon$.
As $T \D \subset \C_{{\epsilon_0}}$, the linear relaxation $LP_{c}(a^t,b)$ has a unique optimal solution with nonbasic variables indexed by $\sigma$  for  any $b\in \sg(a^t)$. By (\ref{gap_via_gamma}), the inequality (\ref{optimality}) holds for $A=a^t$.

Finally, we will show that for some $b'\in \sg(A)$ the optimal value of $IP_c(A,b')$ is equal to $\gap(\Lambda(A),l)+ c_{\bar\sigma}A_{\bar\sigma}^{-1}b'$.
Suppose $\gap(\Lambda(A),l)=m(\Lambda(A),l, r_0)$ and the latter minimum is attained at some $x_0\in \Z^{k}_{\ge 0}$. Then we can equivalently write $\gap(\Lambda(A),l)=m(\Lambda(A),l, x_0)$.
Let us take any vector $u\in \Z_{\ge 0}^{k+1}$ with $u_{\sigma}=x_0$. By Theorem 3 in \cite{Gomory_polyhedra}, $b'=Au$ satisfies the desired property.

\section{Proof of Theorem \ref{lemma_upper_bound}}

Let us find the inradius of the simplex $\Delta_l$. The volume $\vol_k(\Delta_l)= 1/(k!\prod_{i=1}^k l_i)$ and the surface area
\bea
A_{k-1}(\Delta_l)= \sum_{i=1}^k \frac{1}{(k-1)! \prod_{j=1\,,j\neq i}^k l_j} + \frac{|l|}{(k-1)!\prod_{i=1}^k l_i}
=\frac{ \sum_{i=1}^k l_i + |l|}{(k-1)! \prod_{i=1}^k l_i}\,.
\eea
All facets of $\Delta_l$ are touched by the insphere.  Hence, the inradius $r(\Delta_l)$ of the simplex $\Delta_l$ is given by
\be
r(\Delta_l)=\frac{k \,\vol_k(\Delta_l)}{A_k(\Delta_l)}=\frac{1}{\sum_{i=1}^k l_i + |l|}\,.
\label{simplex_inradius}
\ee
Let $B^k(r,x)$ denote the ball in $\R^k$ of radius $r$ centered at $x$. Then, as the covering radius is independent of translation, we have
\be
\rho(\Delta_{l}, \Lambda)\le \rho(B^k(r(\Delta_l),0), \Lambda))= (r(\Delta_l))^{-1}\rho(B^k(1,0),\Lambda)\,.
\label{radius_change}
\ee
Let $\lambda_1, \ldots, \lambda_k$ be Minkowski's successive minima of $B^k(1,0)$ with respect to the lattice
$\Lambda$. Since  $ \Lambda\subset \Z^k$,  we have $\lambda_i\ge 1$ for each $i$.
By Jarnik's inequalities (see e.g. \cite{GrLek})
\be
\rho(B^k(1,0),\Lambda)\le \frac{k \lambda_k}{2}\,.
\label{Jarnik}
\ee

In the geometry of numbers it is customary to use the Hermite constant $\gamma_k$ defined as the lower bound of the constants $\gamma_k'$ such that
every positive definite quadratic form $\sum f_{ij} x_ix_j$ in $k$ variables represents a number $\le \gamma_k' |\det(f_{ij})|^{1/k}$.
It is known (see e.g. Section IX.7. in \cite{Cassels}) that the {\em critical determinant} of $B^k(1,0)$ is equal to $\gamma_k^{-k/2}$.
Therefore, by Minkowski's second theorem for spheres (cf.~\cite[\S 18.4, Theorem 3]{GrLek}), we get
\be
\lambda_k\le \lambda_1\cdots\lambda_{k-1}\lambda_k\le \gamma_k^{k/2} \det(\Lambda)\,.
\label{Minkowski_for_spheres}
\ee
By Lemma \ref{gap_via_simplex}, $\gap(\Lambda,l)=\rho(\Delta_{l}, \Lambda)-\sum_{i=1}^{k} l_i$. Therefore, combining
(\ref{radius_change}), (\ref{simplex_inradius}), (\ref{Jarnik}) and (\ref{Minkowski_for_spheres}) we obtain
the upper bound (\ref{upper_bound}).

\noindent{\bf Acknowledgement}.
The author is grateful to Professor Martin Henk and to the reviewer for useful comments and suggestions.

\end{document}